\documentclass[12pt]{amsart}

\usepackage{amssymb,amsfonts}
\usepackage{amsmath,amsthm,amscd}
\usepackage{mathrsfs} %defines \mathscr{ }
\usepackage{enumerate}
\usepackage{verbatim}
\usepackage{color}
\usepackage[colorlinks]{hyperref}
\usepackage[displaymath, tightpage]{preview}

\usepackage{beton}
\usepackage[T1]{fontenc}%\usepackage{t1enc}
\usepackage[utf8x]{inputenx} %input is in utf8

\newcommand{\stone}{\mathrm{ult}}

\newcommand{\con}{\mathfrak c}

\newcommand{\eps}{\varepsilon}

\newcommand{\alg}{\mathfrak A}

\newcommand{\EE}{{\mathcal E}}

\newcommand{\DD}{{\mathcal D}}

\newcommand{\NN}{{\mathcal N}}

\newcommand{\er}{\mathbb R}

\newcommand{\cov}{{\rm cov}}

\newcommand{\sm}{\setminus}

\newcommand{\sub}{\subseteq}
\newtheorem{thm}{Theorem}[section]

\newtheorem{lem}[thm]{Lemma}
\newtheorem{cor}[thm]{Corollary}
\newtheorem{prop}[thm]{Proposition}

\newtheorem{prob}[thm]{Problem}

\theoremstyle{remark}

\begin{document}
%\makeatletter
%    \providecommand\@dotsep{5}
%  \makeatother
%  \listoftodos\relax

\title{Measures on Suslinean spaces}


 \thanks{The  authors were partially supported by NCN grant 2013/11/B/ST1/03596 (2014-2017).}

\author[Piotr Borodulin--Nadzieja]{Piotr Borodulin-Nadzieja}
\author[Grzegorz Plebanek]{Grzegorz Plebanek}
\address{Instytut Matematyczny, Uniwersytet Wroc\l awski, pl. Grunwaldzki 2/4, 50-384 Wroc\l aw, Poland}
\email{pborod@math.uni.wroc.pl}
\email{grzes@math.uni.wroc.pl}
\date{November 15, 2015}

\begin{abstract}
  We study the existence of non-separable compact spaces that support a measure and are small from the topological point of view. In particular, we show
  that under Martin's axiom there is  a non-separable compact space supporting a measure which has countable $\pi$-character and which cannot be mapped continuously onto $[0,1]^{\omega_1}$.
  On the other hand, we prove that in the random model
  there is no non-separable compact space having countable $\pi$-character and supporting a measure.

\end{abstract}

\maketitle

\section{Introduction} \label{introduction}
The well known Suslin Hypothesis can be expressed in the following form:
\begin{quote}
{\em Every linearly ordered compact space satisfying the countable chain condition is separable.}
\end{quote}

We shall call an arbitrary non-separable {\em ccc} compact space a \emph{Suslinean space}.
Clearly, Suslinean spaces do exist in $\mathsf{ZFC}$. Perhaps the most obvious example is $[0,1]^\kappa$ with the standard product topology, where $\kappa>\mathfrak{c}$.
Such a space, however, is far from being linearly ordered. Indeed,
linearly ordered spaces cannot be mapped continuously onto $[0,1]^{\omega_1}$.

The Suslin Hypothesis was shown to hold under $\mathsf{MA}_{\omega_1}$. Later it turned out that
linearity of the space in question can be relaxed to some kind of topological \emph{smallness}.
For instance, Hajnal and Juhasz (\cite{Juhasz}) proved that under $\mathsf{MA}_{\omega_1}$ every ccc compact space with a $\pi$-base of size less than $\mathfrak{c}$ is separable. By purely topological methods one can deduce from this result that under
$\mathsf{MA}_{\omega_1}$ there are no Suslinean first-countable spaces or Suslinean spaces of countable tightness (see Tall \cite{tall}).

Let us recall that no space which is first-countable or has  countable tightness can be continuously mapped onto $[0,1]^{\omega_1}$.
In the light of the results mentioned above  it is natural to ask if the following statement is consistent:

\begin{quote}
\em
Every
Suslinean space  can be mapped continuously onto $[0,1]^{\omega_1}$.
\end{quote}

 In \cite{Todorcevic} Todor\v{c}evi{\'c} called such a conjecture ``the ultimate version of
Suslin hypothesis''.
Another question which suggests itself here (cf. \cite{Juhasz2}) is:

\begin{quote}
\em
Is it consistent that there is no Suslinean space of countable $\pi$-character?
\end{quote}

Both the questions have been answered negatively. First, Bell (\cite{Bell}) proved that $\mathsf{MA}_{\omega_1}$ is consistent with the existence of a Suslinean space of countable $\pi$-character which does not map continuously onto $[0,1]^{\omega_1}$. Then Moore (\cite{Moore}) gave an example of such space whose
existence is implied by $\mathsf{MA}$. Finally,  Todor\v{c}evi{\'c} (\cite{Todorcevic})
constructed in $\mathsf{ZFC}$ a compact space with the same properties.

The aim of this article is to consider similar questions for the class of compacta that
support a measure. In particular, we discuss the following two questions.

\begin{prob} \label{problem}
	Is there a Suslinean space supporting a measure which cannot be mapped continuously onto $[0,1]^{\omega_1}$?
\end{prob}

\begin{prob} \label{problem2}
	Is there a Suslinean space of countable $\pi$-character supporting a measure?
\end{prob}

 We say that a space $K$ \emph{supports a measure} $\mu$ if $\mu(U)>0$ for every nonempty open $U\subseteq K$ (in other words, $\mu$ is \emph{strictly positive} on $K$).
 Clearly, if a compact space supports a measure then it is
ccc.

Assuming the continuum hypothesis Kunen \cite{Kunen} constructed a first-countable Suslinean space supporting a measure so the answer to \ref{problem} and \ref{problem2} is positive under CH.
Later Kunen and van Mill \cite{Kunen-vanMill} proved the following (cf.\ Plebanek \cite{Pl97}, Theorem 5.1).

\begin{thm} \label{kunen}
The following are equivalent

\begin{itemize}
\item[(i)]
there is a Suslinean first-countable Corson compact space supporting a measure;
\item[(ii)] $\cov(\NN_{\omega_1})=\omega_1$.\footnote{See Section \ref{preliminaries} for the notation
	used here and Section \ref{pi-weight} for further discussion.}
\end{itemize}
\end{thm}

For a definition of a Corson compact space see e.g. \cite{Kunen-vanMill};
it is worth-recalling that every separable subspace of a Corson compactum is metrizable.

We present below a consistent negative solution to Problem \ref{problem2}:  we show that under a certain axiom (satisfied e.g. in the standard random model) there is no Suslinean space of countable
$\pi$-character and supporting a measure.

On the other hand, in Section \ref{example-ma} assuming $\mathsf{MA}$ we construct a space $K$
giving  a positive solution
to both Problem \ref{problem} and Problem \ref{problem2},
i.e.\ our space $K$ is a non-separable compact space supporting a measure and of countable $\pi$-character.
The space $K$ is moreover linearly fibered, i.e. it can be mapped continuously onto $2^\omega$ in such a way that all fibers are linearly ordered. Hence,
 it is in a sense a direct generalization of the Suslin line.
 In particular, $K$ cannot be mapped continuously onto $[0,1]^{\omega_1}$. Moreover, using slightly different methods, we show that under a weaker version of Martin's Axiom there is a small compact space which supports a measure but not
 a countably determined one (so, in a sense, it is close to being non-separable).

We do not know if Problem \ref{problem} can be resolved in the usual set theory.
However,  Theorem \ref{kunen} and Theorem \ref{main} indicate that \emph{usually} the answer to \ref{problem} is positive: if we seek for a consistent negative answer, we have to assume
Martin's axiom for the measure algebras (that is, $\cov(\NN_{\omega_1})>\omega_1$) and the
negation of $\mathsf{MA}_{\omega_1}$, so the standard model obtained by adding $\omega_2$ random reals might be a natural candidate.

The above results seem to reveal an interesting phenomenon. Usually Martin's Axiom eliminates pathological objects (in particular various kinds of Suslinean spaces) and it has been used for this purpose from the very beginning of its creation.
However, the space from Theorem \ref{main}
is a pathology tolerated by Martin's Axiom but not by the classical random model. Perhaps this is an indication that the random model is the universe in which measures on compact spaces behave more orderly than elsewhere.

It seems that under additional axioms the methods used in \cite{Todorcevic} can provide yet another example of small non-separable space supporting a measure. This is a subject of further investigations (see \cite{slaloms}).

\section{Preliminaries} \label{preliminaries}

We use the standard terminology and notation concerning topology, so $w(K)$
is the topological weight of a space $K$ while $\pi(K)$ is the $\pi$-weight, the smallest cardinality of a $\pi$-base of $K$.

Let $K$ be a compact space and $x\in K$. Recall that
a family of nonempty open sets $\mathcal{P}$ is a \emph{$\pi$-base at} $x$ if
for every open $V\ni x$ there is $U\in \mathcal{P}$ such that $U\subseteq V$.
By $\pi(x, K)$ we denote the smallest cardinality of a $\pi$-base at $x$.
The $\pi$-character of $K$ is defined as $\pi\chi(K)=\sup\{\pi(x,K)\colon x\in K\}$.

Given a  Boolean algebra $\mathfrak{A}$,  a family $\mathcal{P}\sub \alg^+$ is its $\pi$-base if for every $A\in\mathfrak{A}^+$ there is  $B\in \mathcal{P}$ with  $B\leq A$. Clearly, every $\pi$-base of a Boolean algebra defines a $\pi$-base of its Stone space.

If $\mathcal{I}$ is an ideal on a set $K$, then
\[ \mathrm{add}(\mathcal{I}) = \min\{|\mathcal{A}|\colon \mathcal{A}\subseteq \mathcal{I}, \ \bigcup \mathcal{A}\notin \mathcal{I}\}, \]
\[ \mathrm{non}(\mathcal{I}) = \min\{|X|\colon X\subseteq K, \ X\notin \mathcal{I}\}, \]
\[ \mathrm{cov}(\mathcal{I}) = \min\{|\mathcal{A}|\colon \mathcal{A}\subseteq \mathcal{I}, \ \bigcup \mathcal{A} = K\}. \]

By a {\em measure} on a topological  space $K$ we mean a finite Borel measure which is Radon, i.e.\ inner-regular with respect to compact sets. Typically, we consider Radon measures on compacta.

If $\mu$ is a measure on $K$ and $\mathcal{N}_\mu$ is the ideal of $\mu$-null sets, then we write $\mathrm{non}(\mu) = \mathrm{non}(\mathcal{N}_\mu)$ and $\mathrm{cov}(\mu) = \mathrm{cov}(\mathcal{N}_\mu)$.

By a measure $\mu$ on a Boolean algebra we mean a finite and  \emph{finitely additive}
function $\mu\colon\alg\to\er$ (note that usually we consider  Boolean algebras which are not $\sigma$-complete).

Let $\alg$ be  a Boolean algebra and let $K=\stone(\alg)$ be its Stone space
(of all ultrafilters). If $A\in\alg$, then by $\widehat{A}=\{x\in K\colon A\in x\}$ we denote the corresponding clopen subset of $K$.
Recall that a measure $\mu$ on $\alg$ can be
transferred to the measure $\widehat{\mu}$ on the algebra of clopen subsets of $K$ via the formula
$\widehat{\mu}(\widehat{A})=\mu(A)$. In turn, $\widehat{\mu}$
can be uniquely extended to a
($\sigma$-additive and Radon) measure on $K$.

 If $\mu$ is a measure on $K$, then by $\mathfrak{A}(\mu)$ we denote the \emph{measure
algebra} of $\mu$, i.e. $\mathfrak{A}(\mu) = \mathrm{Bor}(K)/\{A\colon \mu(A)=0\}$.

Let $\kappa$ be a cardinal number. Denote by $\lambda_\kappa$ the standard product measure on $2^\kappa$ (we also write $\lambda$ for $\lambda_\omega$). Recall that for $\kappa>\omega$, although the product $\sigma$-algebra $\Sigma$ of $2^\kappa$ is
much smaller than the family of Borel sets, by the classical Kakutani theorem,
$Bor(2^\kappa)$ lies in the completion of $\Sigma$ with respect to $\lambda_\kappa$.
Therefore, $\lambda_\kappa$ is in fact defined for all Borel sets.
Let $\mathcal{N}_\kappa$ be the ideal of $\lambda_\kappa$-null subsets of $2^\kappa$. Denote by $\mathfrak{A}_\kappa$ the measure algebra of
$\lambda_\kappa$.

Let $\mu$ the Radon measure on a space $K$.
Let us recall that the {\em Maharam type} of $\mu$ is the least cardinal number $\kappa$ such that there exists a
family $\mathcal{A}$ of Borel sets of size  $\kappa$ approximating $\mu$ with respect to symmetric difference, that is
\[ \inf\{\mu(A\bigtriangleup B)\colon A\in
	\mathcal{A}\}=0\]
for every Borel set $B$.

By the  Maharam structure theorem,
if a Radon measure $\mu$ of of Maharam type $\le\kappa$ then the algebra $\mathfrak{A}(\mu)$
embeds into $\mathfrak{A}_\kappa$, the measure algebra of
$\lambda_\kappa$. Moreover, if $\mathfrak{A}$ is \emph{homogeneous}, i.e.\
the measure $\mu_{| A}$ has the same Maharam type for each nonzero $A\in \mathfrak{A}$,
then  $\mathfrak{A}(\mu)$ and $\mathfrak{A}_\kappa$ are isomorphic.

\section{Suslinean spaces supporting a measure under $\mathsf{MA}$} \label{example-ma}

The main aim of this section is to give a partial positive  answer to Problems \ref{problem} and \ref{problem2}.
 The space $K$ we construct under $\mathsf{MA}$ has the following property:  there is a continuous mapping $f\colon K \to 2^\omega$ such that its fibers are homeomorphic to ordinal numbers
(in particular, they are linearly ordered and scattered). Recall that no scattered compact space can be mapped onto $[0,1]$  by a continuous function. Consequently, by the
proposition given below, our space $K$
does not admit a continuous surjection onto $[0,1]^{\omega_1}$.

\begin{prop}(\cite{Tkachenko}, see also \cite{Drygier})
	Assume $K$ is a compact space, $M$ is compact and metric and $f\colon K \to M$ is a continuous mapping. If $K$ can be mapped continuously onto $[0,1]^{\omega_1}$, then there is $t\in M$ such that $f^{-1}[t]$ can be mapped continuously onto $[0,1]^{\omega_1}$.
\end{prop}

We shall use the following consequence of Martin's Axiom.

\begin{lem}(\cite[33D]{Fremlin}) \label{sequences}
Assume $\mathsf{MA}$. Let $\mu$ be a measure on a compact space $K$ such that $\mu$ has countable Maharam type.
Let $\alpha<\mathfrak{c}$ and let $\{E^\xi_n\colon \xi<\alpha, n\in\omega\}$ be a family of closed
	subsets of $2^\omega$ such that for each $\xi<\alpha$ the sequence $(E^\xi_n)_n$ is increasing and $\lim_n \lambda(E^\xi_n)=1$. Then for each $\varepsilon>0$ there is a closed set
$E\subseteq 2^\omega$ such that
	\begin{enumerate}[(i)]
		\item $\lambda(E)>1-\varepsilon$,
		\item the set $\{i\colon E\nsubseteq E^\xi_i\}$ is finite for every $\xi$.
	\end{enumerate}
\end{lem}

We shall construct the required space $K$ as a Stone space $\stone(\alg)$ of a certain Boolean subalgebra $\alg$ of Borel subsets of $2^\omega$.

\begin{thm} \label{main} Assume $\mathsf{MA}$. There is a compact non-separable space $K$ supporting a measure and a continuous mapping $f\colon K\to 2^\omega$ such that each fiber of $f$ is homeomorphic to an ordinal number and, consequently, $K$ does not map continuously onto $[0,1]^{\omega_1}$. Moreover, $K$ has countable $\pi$-character. \end{thm}
\begin{proof}
	If $F$, $G$ are closed subsets of the Cantor set $2^\omega$, then we write $F\subseteq^* G$ to denote that  for every $x\in G$ there is an open neighbourhood $U\ni x$ such that $F \cap U \subseteq G$ (in other words this means that $F\setminus G$ is closed). Notice that although the relation $\subseteq^*$ is not transitive, if $F_0 \subseteq F_1 \subseteq^* F_2$, then $F_0 \subseteq^* F_2$.

	Fix an enumeration $2^\omega = \{s_\alpha\colon \alpha<\mathfrak{c}\}$. Let $\mathcal{U}$ be the family of clopen subsets of $2^\omega$. Let $\lambda$ be the Lebesgue measure on $2^\omega$. We are going to construct inductively a sequence $(F_\alpha)_{\alpha<\mathfrak{c}}$ of closed subsets
of $2^\omega$ such that for every $\alpha<\mathfrak{c}$
	\begin{enumerate}
		\item $\lambda(F_\alpha)>0$,\label{measure}
		\item $F_\alpha \cap \{s_\xi\colon \xi\leq \alpha\} = \emptyset$,\label{nonseparable}
		\item if $\beta<\alpha$ then $F_\alpha \subseteq^* F_\beta$,\label{fiber}
		\item $\lambda\left(\bigcup_n F_{\alpha+n}\right)=1$.\label{pib}
	\end{enumerate}

	Let $\mathfrak{A}$ be the Boolean algebra of subsets of $2^\omega$ generated by $\mathcal{U}$  and by the family $\{F_\alpha\colon \alpha<\mathfrak{c}\}$. Let $K'=\stone(\mathfrak{A})$
be its Stone space. The measure $\lambda$ restricted to $\mathfrak{A}$ defines a Radon measure
$\mu$ on $K'$ (see Section \ref{preliminaries}).
Let $K\subseteq K'$ be the support of $\mu$, i.e.\
$K=K'\setminus \bigcup\{\widehat{A}\colon \lambda(A)=0\}$.

We shall first prove that the resulting space $K$ has the required properties.
Let $f\colon K \to 2^\omega = \stone(\mathcal{U})$ be the standard continuous mapping induced by clopens of $2^\omega$, i.e. \[ f(x) = \{C\colon C\in\mathcal{U},  C\in x\} \in \stone(\mathcal{U}).\]

\medskip

\noindent {\sc Claim.} The space $K$ supports a measure and is not separable.
\medskip

Obviously, $K$ supports the measure $\mu$.
Conditions (\ref{measure}) and (\ref{nonseparable}) imply that the space $K$ is
	non-separable. Indeed, given a family $\{x_n\colon n<\omega\}\sub \stone(\alg)$ put
$f(x_n)=t_n$ for every $n$. Then $\{t_n\colon n<\omega\}\cap F_\alpha=\emptyset$
for  $\alpha$ large enough so $\widehat{F_\alpha}\cap K\neq\emptyset$ contains no $x_n$.
\medskip

\noindent {\sc Claim.} Each fiber of $f$ is homeomorphic to an ordinal number $<\con$.
\medskip

Indeed, fix $t\in 2^\omega$ and take $\beta<\alpha$ such that $t\in F_\alpha \cap F_\beta$. Then, by property (\ref{fiber}) there is an open $U\ni t$ such that $F_\alpha \cap U \subseteq F_\beta$, which
	means that $f^{-1}[t] \cap \widehat{F_\alpha} \subseteq f^{-1}[t] \cap \widehat{F_\beta}$. 
It follows that the algebra of clopen subsets of $f^{-1}[t]$ is generated by a well-ordered chain
of length $<\con$ so $f^{-1}[t]$ is clearly homeomorphic to some ordinal number $<\con$.
\medskip

\noindent {\sc Claim.} $\pi\chi(K)=\omega$.
\medskip

Consider $x\in K$ and put
	$\gamma = \sup\{\alpha\colon x\in \widehat{F}_\alpha\}$.
	Then $\gamma<\mathfrak{c}$ by property (\ref{nonseparable}).
We claim that the  family
\[ \mathcal{P}_x = \{\widehat{F}_{\gamma+n}\cap \widehat{U}\colon n\in\omega, \ U\in\mathcal{U}\}, \]
forms a countable $\pi$-base at $x$.

Note first that whenever $t=f(x)\notin F_\eta$ for some $\eta<\con$ then, as $F_\eta$
is closed in $2^\omega$, there is $C\in\mathcal{U}$ such that $t\in C$ and $C\cap F_\eta=\emptyset$.
Then $x\in\widehat{C}$ and $\widehat{C}\cap \widehat{F}_\eta=\emptyset$. This remark implies that
the local base at $x\in K$ consists of sets of the form $\widehat{B}$, where
	\[ B = C\cap \left(F_{\xi_1}\cap \dots \cap F_{\xi_k}\right), \]
	and $C\in\mathcal{U}$, $\xi_i\le\gamma$ for every $i\leq l$ and $\lambda(B)>0$. By (\ref{pib}) there is $n$ such that $F_{\gamma+n} \cap B\cap f[K] \ne \emptyset$. Let $t\in F_{\gamma+n} \cap B\cap f[K]$.
Since $F_{\gamma+n} \subseteq^* F_{\xi_i}$ for $i\leq k$, there is $V\in
\mathcal{U}$ such that $t\in V\sub U$ and $F_{\gamma+n} \cap V \subseteq B$. Then $\widehat{F}_{\gamma+n} \cap \widehat{U} \subseteq \widehat{B}$ and $\widehat{F}_{\gamma+n}\cap \widehat{U} \in \mathcal{P}_x$, as required.
\medskip

To complete the proof,
we shall carry out the inductive construction. For $F_0$ take any closed set such that $\lambda(F_0)>0$ and $s_0 \notin F_0$. Assume that we have constructed $\{F_\xi\colon \xi<\alpha\}$.

Notice that for every $\xi<\alpha$ there is an increasing
	sequence $(E^\xi_n)_n$ of closed sets in $2^\omega$ such that
\begin{itemize}
	\item $\lim_n \lambda(E^\xi_n) = 1$,
	\item $E^\xi_n \subseteq^* F_\xi$.
\end{itemize}
	
	Indeed, for every $n$ there is a (possibly empty) closed set $K_n \subseteq 2^\omega\setminus F_\xi$ such that $\lambda(K_n)>\lambda(K\setminus F_\xi)-1/n$. We can assume
	that the sequence $(K_n)_n$ is increasing. Then $E^\xi_n = F_\xi \cup K_n$ is as desired, since $E^\xi_n \setminus F_\xi = K_n$ is closed.

	Now the family $\{E^\xi_n\colon \xi<\alpha, n\in\omega\}$ fulfils the assumptions of Lemma \ref{sequences}. Let $n$ be such that $\alpha = \gamma+n$ for some limit ordinal $\gamma$. Let $E$ be the set given by Lemma \ref{sequences} for
	$\varepsilon = 1/(n+1)$.
As $\mathrm{non}(\mathcal{N})=\mathfrak{c}$ by the Martin's axiom, the set
$\{s_\xi\colon \xi\le\alpha\}$ is of measure zero,  so we can find a closed set $F_\alpha$ such that  $F_\alpha \subseteq E\sm \{s_\xi\colon \xi\le\alpha\}$
and $\lambda(F_\alpha)>1-1/(n+1)$. It is then
enough to check condition (\ref{fiber}). Let $\beta<\alpha$ and $x\in F_\beta$. As $F_\alpha \subseteq E$, there is $k\in \omega$ such that $F_\alpha \subseteq E^\beta_k$. Hence, $F_\alpha \subseteq^* F_\beta$, and we are done.
\end{proof}

Using a weaker version of Martin's axiom we can construct an example of a space with slightly weaker properties. Recall that a measure $\mu$ on a space $K$ is \emph{countably determined} if there is a countable family $\mathcal{F}$ of closed subsets of $K$
such that
\[ \mu(U) = \sup\{\mu(F)\colon F\in \mathcal{F}, F\subseteq U\} \]
for every open $U\subseteq K$.
Clearly, if $K$ supports a countably determined measure, then it is separable. Also, separable spaces support countably determined measures (since purely atomic measures are countably determined).
On the other hand, $2^\con$ is separable but the product measure on $2^\con$ is not countably determined.

Recall that the assertion $\mathfrak{p}=\mathfrak{c}$, which is equivalent to
$\mathsf{MA}(\sigma$-centered), that is Martin's axiom for $\sigma$-centred
 posets, says that every family $\mathcal{A}$ of subsets of $\omega$ of size less than
$\mathfrak{c}$ such that every finite intersection of elements of $\mathcal{A}$ is infinite has an infinite pseudo-intersection $P$ (i.e. $P\setminus A$ is finite for each $A\in \mathcal{A}$).

\begin{thm}  \label{countably-determined}
	Assume $\mathfrak{p}=\mathfrak{c}$. There is a compact space $K$ supporting a measure which is not countably determined and a continuous mapping $f\colon K\to 2^\omega$ such that each fiber of $f$ is homeomorphic to an ordinal number (so, in particular,  $K$ does not map continuously onto $[0,1]^{\omega_1}$).
\end{thm}

\begin{proof}

Let $\lambda$ be the usual measure on $2^\omega$.	
Fix an enumeration $(H_\alpha)_{\alpha<\con}$ of all closed subsets of $2^\omega$ with
$\lambda(H_\alpha)>0$.

Consider $\kappa\le\con$ and  a sequence $(F_\alpha)_{\alpha<\kappa}$ of closed subsets of $2^\omega$ such that for every $\alpha<\kappa$
	\begin{enumerate}
		\item $\lambda(F_\alpha)>0$,\label{w1}
		\item if $\beta<\alpha$ then $F_\alpha \subseteq^* F_\beta$,\label{w2}
        \item $F_\alpha\sub H_\alpha$. \label{w2.5}
	\end{enumerate}

Here we follow the notation used in the proof of
	Theorem \ref{main}.

We again let $\alg$ be the Boolean algebra generated by
$\{F_\alpha\colon \alpha<\kappa\}$ and the family $\mathcal{U}$ of clopen subsets of $2^\omega$. Let $\mu$ be the measure on $K$
uniquely determined by $\mu(K\cap\widehat{A})=\lambda(A)$ for $A\in\alg$.
\medskip

Assume that $\mu$ is countably determined and let $\mathcal{D}$ be the countable family of closed subsets of $K$ determining the measure $\mu$. We can assume that $\mathcal{D}$ is closed under finite unions and intersections.
Let $\mathcal{E} = \{f[D]\colon D\in \mathcal{D}\}$.
\medskip

\noindent{\sc Claim A}. $\kappa<\con$.

Indeed, otherwise $\mathcal{E}$ is a countable family of closed subsets of $2^\omega$ and
every closed $H\subseteq 2^\omega$ with $\lambda(H)>0$ contains some $E\in\mathcal{E}$, but this is plainly
impossible.
\medskip

\noindent{\sc Claim B}.
There is a closed set $F_\kappa\subseteq H_\kappa$ such that $\lambda(F_\kappa)>0$ and $F_\kappa\subseteq^* F_\alpha$ for every $\alpha<\kappa$.
\medskip

To verify Claim B we first prove the following.
\medskip

\noindent {\sc Claim C.} For every $\xi<\kappa$ there is  a sequence $(E^\xi_n)_n$ of elements of $\mathcal{E}$ such that
\begin{enumerate}[(i)]
	\item $\lim_n \lambda(E^\xi_n) = 1$, \label{w3}
	\item $E^\xi_n \subseteq^* F_\xi$. \label{w4}
\end{enumerate}
\medskip

Indeed, fix $\xi<\kappa$.
Notice that both $\widehat{F_\xi}\cap K$ and $ K\setminus f^{-1}[F_\xi]$ are open subsets of $K$. Moreover
\[\mu\left(\widehat{F}_\xi \cup (K\setminus f^{-1}[F_\xi])\right)=1.\]
Therefore for any $\eps>0$ we can find $D_1,D_2\in\DD$ such that
$D_1\sub \widehat{F}_\xi$, $D_2\sub K\setminus f^{-1}[F_\xi]$ and $\mu(D_1\cup D_2)>1-\eps$.
Then $E=f[D_1\cup D_2]\in \EE$, $\lambda(E)>1-\eps$  and $E\sub^* F_\xi$.

Using this observation we can define $E^\xi_n$ satisfying (\ref{w3}) and (\ref{w4}) of Claim C.
\medskip

Now we can prove Claim B.
Let us fix an  enumeration $\mathcal{E} = \{E_n\colon n\in\omega\}$.
For $\xi<\alpha$ and $k>0$ let
\[ T_\xi = \{n\in\omega\colon  \exists k\ E_n\subseteq E^\xi_k\},\quad
N_k=\{n\in\omega\colon \lambda(E_n)>1-1/k\}.\]

Notice that if $I\subseteq \kappa$ is finite, then for every $k$
\[|\bigcap_{\xi\in I}T_\xi\cap N_k| = \omega,\]
since $\mathcal{D}$  is closed under finite intersections and, by (\ref{w3}),
\[\lim_n \lambda\left(\bigcap_{\xi\in I} E^\xi_n\right)=1.\]

Now, since $\mathfrak{p}=\mathfrak{c}$, there is an infinite $T$ such that $T\setminus T_\xi$ is finite for every $\xi<\alpha$ and $T\setminus N_k$ is finite for every $k$. In particular, $\lambda(E_n)\to 1$ for $n\in T$
so we can pick $T_0\sub T$ with the property that
\[ \lambda\left(F\right)> 1-\lambda(H_\kappa), \quad \mbox{where}\quad
 F = \bigcap_{n\in T_0}  E_n.\]
  Now for every $\xi<\kappa$ there is $k$ such that $F\subseteq E^\xi_k$.
 Since $E^\xi_k\sub^* F_\xi$ we get $F\sub^* F_\xi$.
 We put $F_\kappa=F\cap H_\kappa$. Then $\lambda(F_\kappa)>0$ and
 $F_\kappa\sub F\sub^* F_\xi$ for every $\xi<\kappa$, as required.

Claim A and B imply that there must be $\kappa<\con$ such that $\mu$ is not countably determined,
and the proof is complete.
\end{proof}

It is not clear for us if
the space $K$ from Theorem \ref{countably-determined} can be constructed without additional set-theoretic  assumptions.
There is, however, a $\mathsf{ZFC}$ example of
a compact space $K$ which cannot be mapped continuously onto $[0,1]^{\omega_1}$ and which supports a measure which is  not strongly countably determined. Recall that a measure is \emph{strongly countably determined} if there is a countable family of
closed $G_\delta$ sets $\mathcal{F}$ such that
\[ \mu(U) = \sup\{\mu(F)\colon F\in \mathcal{F}, \ F\subseteq U\}.\]
Indeed, the space constructed in \cite{Bell} is separable (hence, it supports a measure), it does not map continuously onto $[0,1]^{\omega_1}$ and it does not have a countable $\pi$-base (and thus it cannot support a strongly countably determined measure).

\section{Measures on spaces of small $\pi$-weight.} \label{pi-weight}

As we have mentioned in the introduction, under $\mathsf{MA}_{\omega_1}$ every space with $\pi$-weight not exceeding $\omega_1$ is separable. In this short section we prove that for spaces supporting a measure we can relax the set-theoretic assumption.

Recall that countably
determined measures were defined in the previous section.
The following fact is an immediate corollary of \cite[Corollary 32H]{Fremlin}.

\begin{prop} \label{frelin-omega_1}
 Assume $\mathsf{MA}_{\omega_1}$. If $K$ is a compact space with $\pi(K)\le \omega_1$, and $K$ supports a measure of countable Maharam type, then each measure supported by $K$ is countably determined (and, consequently, $K$ is separable).
\end{prop}

In particular, under $\mathsf{MA}_{\omega_1}$, if $K$ does not map continuously onto $[0,1]^{\omega_1}$ and $\pi(K)\le\omega_1$  then every measure supported by $K$ is countably determined. The assumption on $\pi$-weight is essential here as it was demonstrated in Theorem \ref{main}.

We shall prove a result similar to Proposition \ref{frelin-omega_1} assuming
Martin's axiom for measure algebras.
It will be convenient to recall several formulations of such an axiom.
The following fact  is a combination of results due to Fremlin and Cicho\'n,
see \cite[525J]{FremlinMT} and \cite[Section 4]{DzPl04} for details.

\begin{thm}\label{cifr}
The following are equivalent

\begin{enumerate}[(i)]
\item $\cov(\NN_{\omega_1})>\omega_1$,
\item $\omega_1$ is a precaliber of measure algebras,
\item $\omega_1$ is a caliber of Radon measures.
\end{enumerate}
\end{thm}

Here we say that $\omega_1$ is a caliber of Radon measures if for any such a measure
$\mu$ on a compact space $K$ and any family $\{B_\alpha\colon \alpha<\omega\}$ of Borel sets
of positive measure there is $x\in K$ such that $x\in B_\alpha$ for uncountably many $\alpha$.

\begin{cor}\label{cifr2}
If $\cov(\NN_{\omega_1})>\omega_1$ then $\cov(\mu)>\omega_1$ for every
Radon measure $\mu$ defined on some compact space $K$.
\end{cor}

\begin{proof}
Otherwise, $K=\bigcup_{\alpha<\omega_1} N_\alpha$ where $\mu(N_\alpha)=0$. Then
there are $F_\alpha\sub K\setminus \bigcup_{\beta<\alpha} N_\alpha$ with $\mu(F_\alpha)>0$
for $\alpha<\omega_1$, and the family $\{F_\alpha\colon \alpha<\omega_1\}$ witnesses that
$\omega_1$ is not a caliber of $\mu$.
\end{proof}

The next theorem is a slight generalization of \cite[Lemma 3.6]{Kamburelis},
where the result was proved only for Boolean spaces under a stronger assumption
 on  weight rather than  $\pi$-weight.
The proof given below is a modification of the proof of \cite[Theorem 1.3]{Juhasz}.

\begin{thm} \label{kamburelis}
	Assume $\cov(\NN_{\omega_1})>\omega_1$. If $K$ is a compact space supporting a measure and $\pi(K) \leq \omega_1$, then $K$ is separable.
\end{thm}

\begin{proof}
Let $\mu$ be a strictly positive measure on $K$.
 We work in the product space $K^\omega$, and denote by $\nu$ the
 product measure $\prod_{n\in\omega}\mu$, so that
 \[ \nu(B_0\times \dots \times B_n \times K \times K
	\dots) = \mu(B_0)\cdot \ldots \cdot \mu(B_n)\]
for all Borel rectangles.

 For every $P\in \mathcal{P}$ define
 \[ \mathcal{G}_P = \{\pi_n^{-1}[K\setminus P]\colon n\in\omega\}\quad \mbox{and}\quad
  D_P = \bigcap \mathcal{G}_P.\]

 Notice that ${\nu}(D_P)=0$ since $\mu (K\setminus P)<1$.
	
Our assumption and Theorem \ref{cifr}
implies that there is $x\in K^\omega$ such that $x\notin \bigcup_{P\in\mathcal{P}} D_P$.
The set $S=\{\pi_n(x)\colon n\in\omega\}$ is then dense in $K$.
Indeed, for an open $U\subseteq K$ there is $P\in \mathcal{P}$ such that $P\subseteq U$. As
$x\in K^\omega\setminus D_P$ there is $n\in\omega$ such that $\pi_n(x)\in P$.
\end{proof}

Let us also state a corollary to a result due to Todor\v{c}evi\'{c} \cite{To90}
(where the notion of free sequences of pairs and its connection with the tightness
is explained).

\begin{cor}
Assuming $\cov(\NN_{\omega_1})>\omega_1$
every countably tight compact space supporting a measure has a countable $\pi$-base.
\end{cor}
\begin{proof}
	\cite[Lemma 1]{To90} states that in every compact space $K$ there is
a family $\{(F_t, G_t)\colon t\in T\}$ such that

\begin{enumerate}[(a)]
\item $F_t\sub G_t\sub K$, where $F_t$ is closed and $G_t$ is open for every $t$,
\item the interiors of $F_t$ are nonempty and  form a $\pi$-base of $K$,
\item for every $T_0\sub T$ the subfamily $\{(F_t, G_t)\colon t\in T_0\}$  is free if and only
if $\bigcap_{t\in T_0} F_t\neq\emptyset$. \label{cefree}
\end{enumerate}

It is enough to note that in our setting $T$ must be countable. Indeed, otherwise take a measure $\mu$ which is strictly positive on $K$. Then $\mu(F_t)>0$ so by Theorem \ref{cifr} and (\ref{cefree}) we get an uncountable free family, which implies that the tightness of K is uncountable.
\end{proof}

\section{Measures on spaces of small $\pi$-character.} \label{pi-character}

In this section we give a relatively consistent negative answer to Problem \ref{problem2}:
there may be  no Suslinean space of countable $\pi$-character supporting a measure.

By ($\Re$) we  denote the following statement: \[ \mathfrak{c}=\omega_2 \ \& \ \mathrm{non}(\mathcal{N})=\omega_1 \ \& \ \mathrm{cov}(\mathcal{N}_{\omega_2}) = \mathfrak{c}.\]  ($\Re$) holds e.g. in the model obtained by forcing with
$\mathfrak{A}_{\omega_2}$. Since $\mathrm{cov}(\mathcal{N}_{\omega_1}) \ge \mathrm{cov}(\mathcal{N}_{\omega_2})$, the
axiom ($\Re$) is stronger that the assumption $\mathrm{cov}(\mathcal{N}_{\omega_1})>\omega_1$
used in the previous section.

The axiom ($\Re$) has an interesting impact on several properties of measures on topological spaces.
Let us mention the following two results proved in \cite{Pl97} and \cite{Pl00}, respectively.

\begin{thm}

Assume($\Re$).

\begin{enumerate}[(a)]
\item A compact space $K$ carries a measure of Maharam type $\con$ if and only if
there is a continuous surjection from $K$ onto $[0,1]^\con$.
\item Every Radon measure on a first-countable compactum is strongly countably determined.
\end{enumerate}
\end{thm}

We shall need a theorem due to Fremlin explaining  that $\mathrm{non}$ and $\mathrm{cov}$ are rather cardinal coefficients of measure algebras than of concrete measure spaces.

\begin{thm}(\cite[Theorem 6.13(c)(d)]{Fremlin-mBA}\label{fremlin})
Suppose $\nu_1$, $\nu_2$ are Radon measures. If $\nu_1$ and $\nu_2$ have isomorphic measure algebras, then $\gamma(\nu_1)=\gamma(\nu_2)$, where $\gamma\in \{\mathrm{non},\mathrm{cov}\}$.
\end{thm}

This observation allows us to prove the following.

\begin{lem}\label{noncov}
	Let $\mu$ be a measure on a compact space $K$. Assume that $K$ cannot be mapped continuously onto $[0,1]^{\mathfrak{c}^+}$. If $\mathrm{non}(\mathcal{N})=\omega_1$, then $\mathrm{non}(\mu)=\omega_1$ and if $\mathrm{cov}(\mathcal{N}_{\omega_2}) = \omega_2$, then $\mathrm{cov}(\mu)=\omega_2$.
\end{lem}
\begin{proof}
	By Kraszewski's theorem (see \cite[Corollary 3.11]{kraszewski})
	\[ \mathrm{non}(\mathcal{N}) =\mathrm{non}(\mathcal{N}_{\omega_1}) = \mathrm{non}(\mathcal{N}_{\omega_2}). \]
It is not difficult to see that $\mathrm{cov}(\mathcal{N}_\kappa)\geq \mathrm{cov}(\mathcal{N}_\lambda)$ if $\kappa\leq \lambda$ (see \cite[Fact 4.1]{kraszewski}). Hence  \[ \mathrm{cov}(\mathcal{N}) = \mathrm{cov}(\mathcal{N}_{\omega_1}) =
\mathrm{cov}(\mathcal{N}_{\omega_2}) = \omega_2.\]
Now it is enough to  notice that the Maharam type of $\mu$ is  at most $\mathfrak{c}$. Indeed, by a theorem due to Haydon, if $K$ carries a Radon measure of Maharam type $\mathfrak{c}^+$, then $K$ can be mapped continuously
onto $[0,1]^{\mathfrak{c}^+}$ (see \cite[Theorem 2.4]{Haydon}). But, if $K$ carries a measure of type $\lambda>\mathfrak{c}$, then it carries a measure of type $\mathfrak{c}^+$, see
Fremlin \cite[Section 531]{FremlinMT}.

By the Maharam theorem and Theorem \ref{fremlin}, we are done.
\end{proof}

\begin{prop} \label{non}
	Assume $(\Re)$. If $K$ is a compact space supporting a measure and $K$ cannot be mapped continuously onto $[0,1]^{\mathfrak{c}^+}$, then $\mathrm{d}(K)\leq \omega_1$.
\end{prop}
\begin{proof}
	We will show that there is $X\subseteq K$ such that $|X|\leq \omega_1$ and $\mu^*(X)=1$. As $\mu$ is strictly positive, we will be done. Lemma \ref{noncov} and Theorem \ref{fremlin} imply a priori only that there is $X'\subseteq K$ of size
	$\omega_1$ such that $\mu^*(X')> 0$. But then we can consider the measurable hull of $X'$, repeat the same argument for $\mu_{|K\setminus X'}$ which is still a Radon measure. Proceeding in this manner, at some countable step we obtain a
	(countable) family of sets of size $\omega_1$ whose union has outer measure $1$ in $K$. Define $X$ as this union.
\end{proof}

We shall now prove the main results of the section.

\begin{thm} \label{pi}
Assume $(\Re)$. If $K$ is a compact space supporting a measure and $\pi\chi(K)\leq \omega_1$, then $K$ is separable.
\end{thm}
\begin{proof}
	Notice that if $K$ is a compact space supporting a measure and $\pi\chi(K)\leq\omega_1$, then $K$ cannot be mapped continuously onto $[0,1]^{\mathfrak{c}^+}$. Indeed, it is known that $w(X)\leq \pi\chi(X)^{c(X)}$ for all $T_3$ spaces (see
	\cite{Sapiro}), where $\mathrm{c}(X)$ is the \emph{cellularity} of $X$, i.e. the supremum of sizes of families of pairwise disjoint open subsets of $K$.
Hence $w(K)\leq \omega_1^\omega = \mathfrak{c}$. Moreover, a continuous mapping cannot increase the weight of compact spaces.

 Now it follows from Proposition \ref{non} that $\mathrm{d}(K)\leq \omega_1$.  Fix a dense set $X\subseteq K$ of size $\omega_1$ and then for each $x\in X$ fix a $\pi$-base $\mathcal{F}_x$ at $x$ of size $\omega_1$. Then $\mathcal{P} = \bigcup_{x\in X} \mathcal{F}_x$ is a $\pi$-base of size at most $\omega_1$. Indeed, if $V$ is a nonempty open subset of $K$, then it contains $x\in X$. So $\mathcal{F}_x$ contains a nonempty open $U\subseteq V$.

We conclude that $\pi(K)\le\omega_1$, and
Theorem \ref{kamburelis} gives the separability of $K$.
\end{proof}

\begin{thm} \label{cov}
	Assume $(\Re)$. Let $K$ be a space of countable $\pi$-character, supporting a measure. Then $K$ has a countable $\pi$-base.
\end{thm}

\begin{proof}
	We already know (see the proof of Theorem \ref{pi}) that there is a $\pi$-base $\mathcal{P}$ of $K$ of size at most $\omega_1$.
Enumerate $\mathcal{P} = \{P_\xi\colon \xi<\omega_1\}$. Denote $\mathcal{P}_\alpha = \{P_\xi\colon \xi<\alpha\}$ and
\[Z_\alpha = \{x\in X\colon \mathcal{P}_\alpha \mbox{ contains a }\pi\mbox{-base of }x\}.\]
Notice that $K = \bigcup_\alpha Z_\alpha$ and that $Z_\alpha$ is closed for every $\alpha$.

As $\mathrm{cov}(\mu)>\omega_1$, there is $\alpha<\omega_1$ such that $\mu(Z_\alpha)=1$. So, $Z_\alpha=K$. One can easily check, as above, that $\mathcal{P}_\alpha$ is a (countable) $\pi$-base of $K$.
\end{proof}

Let us recall that by Theorem \ref{kunen}
 if $\mathrm{cov}(\mathcal{N}_{\omega_1})=\omega_1$, then there is a first-countable non-separable compact space supporting a measure. Of course, first-countable spaces have countable $\pi$-character, so we
 need to assume at least $\mathrm{cov}(\mathcal{N}_{\omega_1})>\omega_1$ to prove Theorem \ref{cov}.

We conclude by some remarks on the so called normal measures.
Recall that a Radon measure on a topological space  $\mu$ is \emph{normal} if each nowhere dense set is null with respect to $\mu$. There are several results concerning the existence of normal measures on certain kind of spaces, e.g.\ Fishel and Papert proved
that there is no normal measure on locally connected spaces (see \cite{Papert}). On the other hand, Plebanek
\cite{Pl14}
constructed  an example of a normal measure on a connected space
(which can be made first-countable under CH). In \cite{Zindulka} Zindulka proved that, consistently there is no normal measure on a first-countable locally compact space. Since the support of a normal measure
cannot be separable, Theorem \ref{cov} implies the following.

\begin{cor} \label{zindulka}
Assume $(\Re)$. There is no normal strictly positive measure on a compact space of a countable $\pi$-character.
\end{cor}

In fact, using \cite[Lemma 3.2, Theorem 3.3 and Lemma 3.4]{Zindulka} and Theorem \ref{pi} one can show that consistently there is no locally compact space with a countable $\pi$-character carrying a non-trivial residual measure.

\section{Open problems}

Problem \ref{problem} remains unsolved in $\mathsf{ZFC}$. In particular, we still do not know the answer to the following question.

\begin{prob} \label{problem3}
	Does $(\Re)$ implies that there is no compact non-separable space $K$ supporting a measure and such that $K$ cannot be mapped continuously onto $[0,1]^{\omega_1}$?
\end{prob}

Perhaps the construction from Section \ref{example-ma} can be improved in such a way that it works in $\mathsf{ZFC}$.

\begin{prob}\label{problem4}
	Is there a compact non-separable space supporting a measure and a continuous mapping $f\colon K \to 2^\omega$ such that $f^{-1}[t]$ is scattered for each $t\in 2^\omega$?
\end{prob}

If the answer to the above is ``yes'', then we would have to use different methods than those from Section \ref{example-ma}. In the proof of Theorem \ref{main} we added a condition (\ref{pib}) to ensure that the constructed space has countable $\pi$-character. Even
without adding condition (\ref{pib}) the space in question would have the $\pi$-character at most $\omega_1$ (if $\mathfrak{c}=\omega_2$). So, by Theorem \ref{pi}, it would have to be separable under axiom $(\Re)$. There is a slightly more general reason
why under axiom $(\Re)$ we would have to invent more subtle methods to construct a similar example.

\begin{prop}
	Assume $(\Re)$. Suppose that $K$ is a compact space supporting a measure $\mu$, $M$ is a compact metric space and $f\colon K \to M$ is a continuous mapping such that for each $t\in M$ the set $f^{-1}(t)$ contains a dense set of isolated points
	of size at most $\omega_1$. Then $K$ is separable.
\end{prop}
\begin{proof}
	Fix a countable base $\mathcal{U}$ of $M$. Let $X\subseteq K$ be a set of size at most $\omega_1$ such that $\mu^*(X)=1$ (see Proposition \ref{non}). Let $Y = f[X]$.
	For each $t\in Y$ let $Y_t$ be the set of isolated points of $f^{-1}(t)$. For each $t\in Y$ and $y\in Y_t$ let $P_y$ be an open set such that $P_y \cap f^{-1}(t) = \{y\}$. We claim that
	\[ \mathcal{P} = \{P_y \cap f^{-1}[U] \colon y\in Y_t, t\in Y, U\in\mathcal{U}\} \]
	is a $\pi$-base of $K$ and thus, according to Theorem \ref{pi}, $K$ is separable.

	Let $V\subseteq K$ be an open set. Then there is $x\in V \cap X$. As the set of isolated points is dense in $f^{-1}(f(x))$, there is $y\in Y_{f(x)}$ such that $y\in V$. We claim that there is $U\in \mathcal{U}$ such that $P_y \cap f^{-1}[U]
	\subseteq V$ and $f(x)\in U$ (and so $P_y \cap f^{-1}[U]$ is non-empty).
	Suppose to the contrary that
	\[ P_y \cap f^{-1}[U])\setminus V \ne \emptyset \]
	for each $U\in \mathcal{U}$ such that $f(x)\in U$. Then also
	\[ \overline{P_y} \cap \overline{f^{-1}[U]} \setminus V \ne \emptyset \]
	and so by compactness
	\[ \overline{P_y} \cap f^{-1}(f(x)) \setminus V \ne \emptyset. \]
	But $\overline{P_y} \cap f^{-1}(f(x)) = \{y\}$ and $y\in V$, a contradiction.
\end{proof}

At last, we do not know if one can prove a theorem similar to Corollary \ref{zindulka} for spaces which cannot be mapped onto $[0,1]^{\omega_1}$.

\begin{prob}
	Is it consistent that there is no compact space which cannot be mapped continuously onto $[0,1]^{\omega_1}$ and which supports a normal measure?
\end{prob}

\section{Acknowledgements}
This research was completed whilst the authors were visiting fellows at the Isaac Newton Institute for Mathematical Sciences, Cambridge, in the
programme `Mathematical, Foundational and Computational Aspects of the Higher Infinite' (HIF).
We would like to thank Mirna D\v{z}amonja for interesting discussions on the subject of this paper.

\bibliographystyle{alpha}
\bibliography{smallb}

\begin{thebibliography}{KvM95}

\bibitem[Bel96]{Bell}
Murray Bell.
\newblock A compact ccc non-separable space from a {H}ausdorff gap and
  {M}artin's axiom.
\newblock {\em Comment. Math. Univ. Carolin.}, 37(3):589--594, 1996.

\bibitem[BN15]{slaloms}
Piotr Borodulin-Nadzieja.
\newblock Measure and slaloms.
\newblock {\em Preprint}, 2015.

\bibitem[BND15]{Drygier}
Piotr Borodulin-Nadzieja and Piotr Drygier.
\newblock Measures on fibers.
\newblock {\em Preprint}, 2015.

\bibitem[DP04]{DzPl04}
Mirna D{\v{z}}amonja and Grzegorz Plebanek.
\newblock Precalibre pairs of measure algebras.
\newblock {\em Topology Appl.}, 144(1-3):67--94, 2004.

\bibitem[FP64]{Papert}
B.~Fishel and D.~Papert.
\newblock A note on hyperdiffuse measures.
\newblock {\em J. London Math. Soc.}, 39:245--254, 1964.

\bibitem[Fre84]{Fremlin}
D.~H. Fremlin.
\newblock {\em Consequences of {M}artin's axiom}, volume~84 of {\em Cambridge
  Tracts in Mathematics}.
\newblock Cambridge University Press, Cambridge, 1984.

\bibitem[Fre89]{Fremlin-mBA}
David~H. Fremlin.
\newblock Measure algebras.
\newblock In {\em Handbook of {B}oolean algebras, {V}ol.\ 3}, pages 877--980.
  North-Holland, Amsterdam, 1989.

\bibitem[Fre08]{FremlinMT}
D.~H. Fremlin.
\newblock {\em Measure theory. {V}ol. 5. Set-theoretic Measure Theory.}
\newblock Torres Fremlin, Colchester, 2008.

\bibitem[Hay77]{Haydon}
Richard Haydon.
\newblock On {B}anach spaces which contain {$l^{1}(\tau )$} and types of
  measures on compact spaces.
\newblock {\em Israel J. Math.}, 28(4):313--324, 1977.

\bibitem[Juh71]{Juhasz}
I.~Juh{\'a}sz.
\newblock {\em Cardinal functions in topology}.
\newblock Mathematisch Centrum, Amsterdam, 1971.
\newblock In collaboration with A. Verbeek and N. S. Kroonenberg, Mathematical
  Centre Tracts, No. 34.

\bibitem[Juh77]{Juhasz2}
I.~Juh{\'a}sz.
\newblock {\em Consistency results in topology}.
\newblock North-Holland, 1977.
\newblock Handbook of mathematical logic, editor J. Barwise.

\bibitem[Kam89]{Kamburelis}
Anastasis Kamburelis.
\newblock Iterations of {B}oolean algebras with measure.
\newblock {\em Arch. Math. Logic}, 29(1):21--28, 1989.

\bibitem[Kra01]{kraszewski}
Jan Kraszewski.
\newblock Properties of ideals on the generalized {C}antor spaces.
\newblock {\em J. Symbolic Logic}, 66(3):1303--1320, 2001.

\bibitem[Kun81]{Kunen}
Kenneth Kunen.
\newblock A compact {$L$}-space under {CH}.
\newblock {\em Topology Appl.}, 12(3):283--287, 1981.

\bibitem[KvM95]{Kunen-vanMill}
Kenneth Kunen and Jan van Mill.
\newblock Measures on {C}orson compact spaces.
\newblock {\em Fund. Math.}, 147(1):61--72, 1995.

\bibitem[Moo99]{Moore}
J.~Tatch Moore.
\newblock A linearly fibered {S}ouslinean space under {MA}.
\newblock In {\em Proceedings of the 1999 {T}opology and {D}ynamics
  {C}onference ({S}alt {L}ake {C}ity, {UT})}, volume~24, pages 233--247, 1999.

\bibitem[Ple97]{Pl97}
Grzegorz Plebanek.
\newblock Nonseparable {R}adon measures and small compact spaces.
\newblock {\em Fund. Math.}, 153(1):25--40, 1997.

\bibitem[Ple00]{Pl00}
Grzegorz Plebanek.
\newblock Approximating {R}adon measures on first-countable compact spaces.
\newblock {\em Colloq. Math.}, 86(1):15--23, 2000.

\bibitem[Ple14]{Pl14}
Grzegorz Plebanek.
\newblock A normal measure on a compact connected space.
\newblock {\em Preprint}, 2014.
\newblock http://arxiv.org/abs/1507.02845.

\bibitem[{\v{S}}ap74]{Sapiro}
B.~{\v{S}}apirovski{\u\i}.
\newblock Canonical sets and character. {D}ensity and weight in bicompacta.
\newblock {\em Dokl. Akad. Nauk SSSR}, 218:58--61, 1974.

\bibitem[Tal74]{tall}
Franklin~D. Tall.
\newblock The countable chain condition versus separability---applications of
  {M}artin's axiom.
\newblock {\em General Topology and Appl.}, 4:315--339, 1974.

\bibitem[Tka91]{Tkachenko}
Michael~G. Tka{\v{c}}enko.
\newblock {$\scr P$}-approximable compact spaces.
\newblock {\em Comment. Math. Univ. Carolin.}, 32(3):583--595, 1991.

\bibitem[Tod90]{To90}
Stevo Todor{\v{c}}evi{\'c}.
\newblock Free sequences.
\newblock {\em Topology Appl.}, 35(2-3):235--238, 1990.

\bibitem[Tod00]{Todorcevic}
Stevo Todor\v{c}evi\'{c}.
\newblock Chain-condition methods in topology.
\newblock {\em Topology Appl.}, 101(1):45--82, 2000.

\bibitem[Zin00]{Zindulka}
Ond{\v{r}}ej Zindulka.
\newblock Residual measures in locally compact spaces.
\newblock {\em Topology Appl.}, 108(3):253--265, 2000.

\end{thebibliography}

\end{document}